\documentclass[12pt]{article}

\usepackage[T1]{fontenc}
\usepackage[T2A]{fontenc}
\usepackage[utf8]{inputenc}
\usepackage[russian,english, french]{babel}

\usepackage[mathscr]{eucal}
\usepackage{amsbsy}
\usepackage{amssymb}
\usepackage{amsmath}
\usepackage{amsthm}

\usepackage{latexsym}

\usepackage{stmaryrd}
\usepackage[dvips]{graphicx,epsfig}
\usepackage{color}
\newcommand{\E}{\mathbf{E}\,}

\def\phi{\varphi}
\def\bbn{{\mathbb N}}
\def\bbr{{\mathbb R}}
\def\bbb{{\mathbb B}}

\def\Bl1{{\bf 1}}
\def\B2{{\bf 2}}
\def\B0{{\bf 0}}

\def\di{{\mathrm {dist}}}

\def\e{\varepsilon}

\def\ls{\mathrm {limsup}}
\def\=A8{\"o}

\def\co{\mathop{\hbox{\rm conv}}\nolimits}

\newcommand{\beq}{\begin{equation}}
\newcommand{\eeq}{\end{equation}}
\newcommand\beqn{\begin{displaymath}}  % no number
\newcommand\eeqn{\end{displaymath}}
 
\newcommand{\halmos}{\vspace{3mm} \hfill \mbox{$\Box$}\\[2mm]}
\theoremstyle{plain}
\newtheorem{teo}{Theorem}
\newtheorem{lem}[teo]{Lemma}

\theoremstyle{definition}

\date{}

\begin{document}

{\bf Youri Davydov}
\vspace{15pt}
	
\noindent
	{\Large \bf A counter-example linked to Gaussian convex hulls}
	\vspace{15pt}

% e-mail: youri.davydov@univ-lille.fr   }}

%\vfill
%\eject
\section{Introduction}

%\subsection{Notation}

Consider the sequence  $\{X_n\}$ of  independent centered Gaussian random elements of a separable Banach space $\mathbb{B}$ and their closed convex hulls\\   $W_n = \co \{\,X_{1},\ldots,X_n\},\;n\in \bbn.$

In 1988, under the assumption that  $X_n$  have a common distribution  ${\cal P},$   Goodman \cite{Goodman} obtained a fundamental result on the convergence of the normalized samples
\begin{equation}\label{good}
\frac{1}{b(n)}\{\,X_{1},\ldots,X_n\}, \;\;\; b(t) = \sqrt{2\ln (t)}, \;t>1,
\end{equation}
to the concentration ellipsoid  ${\cal E} $   of the distribution  ${\cal P}$ .

From this result it immediately follows that almost surely, in the \\ Haussdorff metric $d_H$, we have
\begin{equation}
\label{conv1}
\frac{1}{b(n)}W_n {\rightarrow} \;{\cal E}, \quad \ n\to \infty.
\end{equation}
Later this result was extended in different directions:

 - Its analog for random elements from a Skorokhod space was obtained in \cite{Dav}.
 
 - The convergence \eqref{conv1} was proved for stationary weakly dependent Gaussian fields in \cite{DP1}.
 
 - Finally, it was shown in \cite{DP2} that in the weakly dependent case, it was possible to replace the stationarity assumption with an assumption of weak convergence: 
 \begin{equation}\label{conv2}
 X_n  \Rightarrow X.
\end{equation}
 
 In the same work were given examples showing that in the absence of Assumption (\ref{conv2}), the limit set can exist, but its shape can strongly differ from ellipsoidal.  For example, in the case  $\mathbb{B} = \mathbb{R}^d$,       an example of independent Gaussian vectors was given for which  $\frac{1}{b(n)}W_n $
 converged to a given polytope with central symmetry.

The goal of the present note is to show that if the assumption of weak convergence of the initial sequence is relaxed, the limit set can be an arbitrary convex compact set.

\section{The main result}

\begin{teo}\label{1}
Let $V\subset \mathbb{B}$ be a convex compact centrally symmetric set.

Then there exists a sequence of independent Gaussian vectors $\{X_k\}$ such that
$$
\sup_n\{\E|X_n|^2 \}\,<\,\infty,
$$
and, with probability 1,

$$
\frac{1}{b(n)}W_n\;\rightarrow \;V.
$$
\end{teo}
\vspace{5pt}

{\bf Proof.}
Let   $\{ T_k\}$ be a partition of the set of positive integers, $\bbn=\cup_{k=1}^\infty T_k$,  such that  $T_k$  have asymptotic  density   $p_k>0$ and $\sum p_k = 1$ (it means that $\lim_{n} \{ \frac{1}{n}\#\{T_k\cap[1,n]\}\}=p_k$).

Let  $\{s_k\}$ be a countable dense subset of the unit sphere in $\mathbb{B}$ and
$a_k = \sup\{t>0\,|\,ts_k\in V\}.$ Let $\{\xi_k\}$ be a  sequence of independent standard Gaussian random variables. Let
$X_k= a_k\xi_ks_k.$ It is clear that these random vectors will be independent and that the distribution of $X_k$ will be concentrated on the line
 $\{ts_k, t\in \mathbb{R}\},$      and will be Gaussian with mean $0$  and variance $ \sigma^2_j:=|a_k|^2$.   Let us show that for the sequence defined this way $\frac{1}{b(n)}W_n$ will converge to the desired limit.

First, let us show that the sequence  $\{\frac{1}{b(n)}W_n\}$  will be relatively compact with probability 1.

Lemmas 2.2 and 2.3 from   \cite{DP1}    show that it is sufficient to find a compact $K$  such that, for any $\e >0$, with probability 1 for all large enough $n$, we have:
\beq\label{comp1}
\frac{1}{b(n)}W_n \subset K^{\e}.
\eeq
Take $K=V.$   It is easy to see that the inclusion (\ref{comp1})  will take place if for any  $\e >0$   with probability 1,
\beq\label{comp2}
\ls\{\di(X_n,\; (1+\e)b(n)V)\} = 0,
\eeq
where hereafter  $\di(x, A)$ designates the distance between  $x$ and a set $A.$

For that, it is sufficient to show that for any  $\e >0$,
\beq\label{comp3}
\sum_n\mathbb{P}\{\di(X_n,\; (1+\e)b(n)V) >\e\}  \,<\,\infty.
\eeq
It is clear, that for $n\in T_m$,
\begin{align*}
\mathbb{P}\{\di(X_n,\; (1+\e)b(n)V) >\e\}& \leq\;\mathbb{P}\{\di(X_n,\; (1+\e)b(n)[-a_ms_m,\,a_ms_m]) >\e\}\\
&\leq \mathbb{P}\{|X_n| > (1+\e)b(n)|a_m| \}\\
&= \mathbb{P}\{\frac{|X_n|} {|a_m|}> (1+\e)b(n) \}\\
&= \mathbb{P}\{|\xi| > (1+\e)b(n) \},
\end{align*}
where  $\xi$ is a standard Gaussian random variable and $[x,y]$ designates the segment between the two points $x$ and $y$.

Since  $C(\gamma):= \mathbf E \exp\{\gamma |\xi|^2\} < \infty$    for any $\gamma,\, 0<\gamma< 1/2,$
we have that
\beq\
\mathbb{P}\{|\xi| > (1+\e)b(n) \} \leq \frac{{\mathbf E} \exp{\{\gamma |\xi|^2\}}}
{\exp\{\gamma b(n)^2 (1+\e)^2}\; =\; \frac{C(\gamma)}{n^{2\gamma (1+\e)^2}}.
\eeq
In this way, taking  $\frac{1}{2 (1+\e)^2} <\gamma<1/2,$  we obtain the convergence of the series (\ref{comp3}),
which gives through the Borel-Cantelli lemma relative compactness of the sequence  $\{\frac{1}{b(n)}W_n\}.$

Recall that the support function of a convex set $K$ is a real function $M_K$ defined on the unit sphere $S^\star (0,1)$ of the dual space to the space $\mathbb{B}$,
$$
M_K(x^\star): = \sup_{x\in K}\{\langle x, x^\star\rangle \}.
$$

It is a continuous function that characterizes  $K,$  and moreover,  
$$
d_H(K_1,K_2) = \|M_{K_1} - M_{K_2} \|_{sup} 
$$ 
and
$$
M_{ \co(K_1,K_2)} (x^\star) = \max\{M_{ K_1} (x^\star) \;,\;M_{ K_2} (x^\star)\}.
$$

Now, let us consider the support function of the set $\frac{1}{b(n)}W_n,$ denoting it $M_n.$ We have
$$
M_n(x^\star) = \frac{1}{b(n)}\max_{k\leq n}\{\langle X_k,x^\star\rangle\}.
$$
As the points  $\{a_js_j\}$  form a dense subset of the boundary $V,$   the support function of the set  $V$ is equal to
$$
M_V (x^\star) = \sup_j \{a_j\langle s_j,x^\star\rangle\}.
$$

Since for any   $k \in T_j\,$ we have $ \langle X_k,x^\star\rangle\, =   a_j\langle s_j,x^\star\rangle \xi_k,$
for each $m\in \bbn$  such that $n\geq m$,
\begin{align*}
M_n(x^\star) \;&\geq\; \max_ {j\leq m}\left\{ \frac{1}{b(n)}\max_{k\leq n, j\in \cup_j T_j} \{ \langle X_k,x^\star\rangle\}\right\}\\
&= \max_ {j\leq m}\left\{ a_j\langle s_j,x^\star\rangle\frac{1}{b(n)}\max_{k\leq n, j\in \cup_j T_j}\{\xi_k\}\right\}.
\end{align*}
Since $ \frac{b(np)}{b(n)} \to 1$  for any $p>1,$   by (\ref{conv1})
$\frac{1}{b(n)}\max_{k\leq n, j\in \cup_j T_j}\{\xi_k\}\to 1$ almost surely, and therefore with probability 1, for any $m$,
$$
\liminf_n M_n(x^\star) \geq  \max_ {j\leq m}\{ a_j\langle s_j,x^\star\rangle\},
$$
which gives
\beq\label{supp1}
\liminf_n M_n(x^\star) \geq M_V (x^\star).
\eeq
On the other hand, since for $k\leq n,\;k\in T_j $ we have the inequality
$$
 \langle X_k,x^\star\rangle \, =   a_j\langle s_j,x^\star\rangle \xi_k\; \leq\; M_V (x^\star)
\max_{k\leq n}\{|\xi_k|\},
$$
we have that
$$
M_n(x^\star) \leq M_V (x^\star) \max\left\{
 \frac{1}{b(n)} \max_{k\leq n}\{\xi_k\} \;,\; \frac{1}{b(n)} \max_{k\leq n}\{-\xi_k\} \right\},
$$
and therefore, using once again (\ref{conv1}),  with probability 1,
\beq\label{supp2}
\limsup_nM_n(x^\star) \leq M_V (x^\star) 
\eeq

From (\ref{supp1})   and  (\ref{supp2})   it follows that for any  $x^\star \in S^\star (0,1)$, with probability 1, we have that
$$
M_n(x^\star) \to  M_V (x^\star) .
$$

Now we will use the following fact:

%%%%%%%%%%%%%%%%%%%%%%%%%%%%%%%%%%%%%%%%%%%%%%%%%%%%%%%%%%
\begin{lem}\label{L5} (Lemma 2.7 from \cite{DP1})

Let $(K_n)_{n\geq 0}$ be a sequence of random compact convex subsets of
$\bbb$. Assume that $(B_n)$ is a.s. totally bounded. Assume
also that there exists a (deterministic) function
$\phi:S_1^*(0,1)\to\bbr$ such that, for all $\theta \in
S_1^{\ast}(0,1)$,
\[
\mathcal{M}_{K_n}(\theta)\to \varphi(\theta)\quad  \mbox{a.s., as }
n\rightarrow +\infty.
\]
Then $\varphi$ is the support function of a compact convex
 set $A\subset \bbb$  and
\[
K_n\rightarrow A  \quad  \mbox{a.s., as } n\rightarrow +\infty.
\]
\end{lem}

In conjunction with the relative compactness, we finally obtain
 $$
\frac{1}{b(n)}W_n {\rightarrow} \;{V}.
$$

%%%%%%%%%%%%%%%%%%%%%%%%%%%%%%%%%%%%%%%%%%%%%%%%%%%%%%%%%%%%%

\halmos
\vspace{5pt}

{\bf  Remark.} If the second moments of the norms of the initial vectors  $\{X_k\}$   are uniformly bounded, then by changing slightly the reasoning above and using Fernique's theorem (\cite{F}, Theorem 1.3.3), we obtain relative compactness of the sequence  $\{\frac{1}{b(n)}W_n\}.$ So in that case, the limit set for 
 $\{\frac{1}{b(n)}W_n\}$ can only be a compact set.
 \vspace{10pt}
 
{\bf Acknowledgements.}  

This work was performed at the Saint Petersburg international mathematical Leonard Euler institute and  supported by the Ministry of Science and Higher Education  of the Russian federation  (agreement № 075–15–2025–343).
 \vspace{5pt}

{\bf Affiliation.}

Saint Petersburg State University,  7/9 Universitetskaya Embankment,\\  St. Petersburg, 199034 Russia,

\hspace{10pt} and

Lille university, Cite Scietifique, 59650 Villeneuve-d'Ascq, France.
%\cqfd

%\vspace{5pt}

%\halmos

\section*{References}

\bibliographystyle{plain}

{\bf  Abstract}

We consider the sequence  of  independent centered Gaussian random elements of a separable Banach space  and their consecutive closed convex hulls. If inicial elements  converge weakly to some limite, then, as shown in Davydov-Paulauskas (2024), its normalized convex hulls converge,  with probability 1, to the concentration  ellipsoid  of the limiting distribution. 

 The goal of the present note  is  to  show  that  if  the assumption  of  weak
convergence of the initial  sequence  is  relaxed, than  the limit  set can  be  an arbitrary
convex compact set.

\end{document}